\documentclass[11pt]{article}
 \usepackage{amsmath,amsthm,amsfonts,amssymb,bbm}
 \usepackage[sort]{natbib} 
 \usepackage[pagebackref=true,linktoc=page,colorlinks=true,linkcolor=blue,citecolor=blue]{hyperref}

\usepackage{enumerate}
\usepackage{multirow} 
\usepackage{graphicx}
\usepackage[all]{xy}
\usepackage{tabularx} 
\usepackage{booktabs} 
\usepackage{accents} 
\usepackage[font=small,labelfont=bf]{caption} 

\newcommand{\cf}{cf.\ }

\newcommand{\RR}{\ensuremath{\mathbb{R}}}

\newcommand{\PP}{\ensuremath{\mathbb{P}}}

\newcommand{\distr}{\ensuremath{\mathcal{D}}}

\def\ci{\perp\!\!\!\perp}
\newcommand{\D}{\ensuremath{\text{\normalfont d}}}
\newcommand{\subs}{\ensuremath{\mathcal{C}}}

\theoremstyle{plain}
\newtheorem{theorem}{Theorem}
\newtheorem{proposition}[theorem]{Proposition}
\newtheorem{corollary}[theorem]{Corollary}
\newtheorem{lemma}[theorem]{Lemma}

\theoremstyle{definition}

\newtheorem{example}{Example}

\theoremstyle{remark}
\newtheorem*{remark}{Remark}

\begin{document}

\title{Conditional independence among max-stable laws} 

\author{
  Ioannis Papastathopoulos\footnote{School of Mathematics and Maxwell Institute, University of Edinburgh, Edinburgh, EH9 3FD Email: i.papastathopoulos@ed.ac.uk},\\
  Kirstin Strokorb\footnote{Institute of Mathematics, University of
    Mannheim, D-68131 Mannheim, Email: strokorb@math.uni-mannheim.de} }

\maketitle 

\begin{abstract}
Let $X$ be a max-stable random vector with positive continuous density.
It is proved that the conditional independence of any collection of disjoint subvectors of $X$ given the remaining components implies  their joint independence.
We conclude that a broad class of tractable max-stable models cannot exhibit an interesting Markov structure.
\end{abstract}

{\small
\phantom{\textit{2010 MSC}:} {Secondary 62H05}
}


\section{Introduction}\label{sect:intro}

As pointed out by \cite{daw79} \emph{independence} and
\emph{conditional independence} are key concepts in the theory of
probability and statistical inference.  A collection of (not
necessarily real-valued) random variables $Y_{1},\dots,Y_{k}$ on some
probability space $(\Omega,\mathcal{A},\PP)$ are called {conditionally
  independent} given the random variable $Z$ (on the same probability space) 
if
\begin{align*}
  \PP(Y_1\in A_1,\dots,Y_k \in A_k \mid Z) = \prod_{i=1}^k
  \PP(Y_i \in A_i \mid  Z) \qquad \PP\text{-a.s.},
\end{align*}
for any measurable sets $A_1,\dots,A_k$ from the respective state
spaces. The conditioning is meant with respect to the $\sigma$-algebra generated by $Z$.  
A particularly important example for the conditional independence to be
an omnipresent attribute are the \emph{Gaussian Markov random fields}
that have evolved as a useful tool in spatial statistics
\citep{ruehel05,lau96}.  Here, the zeroes of the \emph{precision
  matrix} (the inverse of the covariance matrix) of a Gaussian random
vector represent precisely the conditional independence of the
respective components conditioned on the remaing components of the
random vector.  Hence, sparse precision matrices are desirable for
statistical inference.

In the analysis of the extreme values of a distribution (rather than
fluctuations around mean values) \emph{max-stable} models have been
frequently considered. We refer to
\cite{engetal14,navetal09,bladav11,buisetal08} for some spatial
applications among many others.  Their popularity originates from the
fact that max-stable distributions arise precisely as possible limits
of location-scale normalizations of i.i.d.\ random
elements. 
A random vector $X$ is called {max-stable} if it satisfies the
distributional equality $a_n X + b_n \stackrel{\distr}{=}
\max(X^{(1)},\dots,X^{(n)})$ for independent copies
$X^{(1)},\dots,X^{(n)}$ of $X$ for some appropriate normalizing
sequences $a_n>0$ and $b_n \in \RR$, where all operations are meant
componentwise. 
If the components $X_i$ of $X$
are \emph{standard Fr{\'e}chet} distributed, i.e. $\PP(X_i \leq
x)=\exp(-1/x)$ for $x \in (0,\infty)$, we have $a_n=n$ and $b_n=0$ and
the random vector $X$ will be called \emph{simple max-stable}. 

Let $I$ be a non-empty finite set. It is
well-known (cf.\ e.g.\ \cite{res08}) that the distribution functions
$G$ of simple max-stable random vectors $X=(X_i)_{i \in I}$ are in a
one-to-one correspondence with Radon measures $H$ on some reference
sphere $S_+=\{ \omega \in [0,\infty)^I : \lVert \omega \rVert = 1\}$
that satisfy the moment conditions $\int \omega_i H(\D \omega)=1$, $i
\in I$. The correspondence between $G$ and $H$ is given by the
relation
\begin{align*}
  G(x)=\PP(X_i \leq x_i, \, i \in I) = \exp\left(- \int_{S_+}
    \max_{i \in I} \frac{\omega_i}{x_i}~H(\D \omega)\right), \qquad
  x \in (0,\infty)^I.
\end{align*}
Here, $\lVert \cdot \rVert$ can be any norm on $\RR^I$ and $H$ is
often called \emph{angular} or \emph{spectral measure}.

In general, neither does independence imply conditional independence
nor does conditional independence imply independence of the subvectors
of a random vector.  Consider the following two simple examples which
illustrate this fact in the case of Gaussian random vectors (Example
\ref{ex:gauss}) and max-stable random vectors (Example \ref{ex:max}).
For notational convenience, we write \mbox{$X \ci Y$} if $X$ and $Y$ are
independent and \mbox{$X \ci Y \mid Z$} if $X$ and $Y$ are conditionally
independent given $Z$ and likewise use the instructive notation
\mbox{$\ci_{i=1}^k X_i$} and \mbox{$\ci_{i=1}^k X_i \mid Z$} if more than two random
elements are involved.

\begin{example}\label{ex:gauss}
  Let $X_1,X_2,X_3$ be three independent standard normal random
  variables and, moreover, $X_4=X_1+X_2$ and $X_5=X_1+X_2+X_3$. Then
  all subvectors of $(X_i)_{i=1}^5$ are Gaussian and
  \begin{align}
    \label{eq:indep:example} &X_1 \ci X_2,  && \hspace{-2cm} \text{but not} \qquad X_1 \ci X_2 \mid X_5,\\
    \label{eq:cindep:example} \text{whereas} \qquad &X_1 \ci X_5 \mid
    X_4, && \hspace{-2cm} \text{but not} \qquad X_1 \ci X_5.
  \end{align}
\end{example}

\begin{example} \label{ex:max}
  Let $X_1,X_2,X_3$ be three independent standard Fr{\'e}chet random
  variables and, moreover, $X_4= \max(X_1, X_2)$ and $X_5= \max(X_1 ,X_2,X_3)$.
  Then all subvectors of $(X_i)_{i=1,\dots,5}$ are
  max-stable and both relations \eqref{eq:indep:example} and
  \eqref{eq:cindep:example} hold true also in this setting.
\end{example}

However, if the distribution of a max-stable random vector has a
{positive continuous density}, then conditional independence of any two
subvectors conditioned on the remaining components implies already 
their independence. 
To be precise, when we say that a random vector has a \emph{positive continuous density}, we mean that the joint distribution of its components has a positive continuous density.
The following theorem is the main result of the
present article. If $X=(X_i)_{i \in I}$ is a random vector, we write
$X_A$ for the subvector $(X_i)_{i \in A}$ if $A \subset I$. The same convention
applies to non-random vectors $x=(x_i)_{i \in I}$.

\begin{theorem}\label{thm:CImaxstabledensity}
  Let $X=(X_i)_{i \in I}$ be a simple max-stable random vector with
  positive continuous density. Then, for any disjoint non-empty subsets $A$ and $B$ of $I$,  the conditional independence 
  \mbox{$X_A \ci X_B \mid X_{I \setminus(A \cup B)}$} implies the independence
  \mbox{$X_A \ci X_B$}.
\end{theorem}

A proof of this theorem is given in
Section~\ref{sec:proofs}. Beforehand, some comments are in order.

(a) First, the requirement of a positive continuous density
for $X$ is much less restrictive than requiring the spectral measure
$H$ of $X$ to admit such a density, \cf \cite{beirletal04}
pp.\,262-264 and references therein. For instance, fully independent
variables $X=(X_i)_{i \in I}$ have a discrete spectral measure, while
their density exists and is positive and continuous.  A more subtle
example is, for instance, the asymmetric logistic model \citep{tawn90}, which admits a continuous positive density and whose spectral measure carries mass
on all faces of $S_+$, \cf also Example~\ref{ex:examples}.

(b) Secondly, both random vectors
$(X_i)_{i=1,2,5}$ and $(X_i)_{i=1,4,5}$ that were considered in the
Gaussian case in Example~\ref{ex:gauss} have a positive continuous
density on $\RR^d$. Hence, there exists no version for
Theorem~\ref{thm:CImaxstabledensity} for the Gaussian case.

(c) Note that the implication of Theorem~\ref{thm:CImaxstabledensity} is the independence of $X_A$ and $X_B$, not the independence of all three subvectors $X_A,X_B,X_{I \setminus (A \cup B)}$.

(d) By means of the same argument that shows that pairwise
independence of the components of a max-stable random vector implies
already their joint independence, we may deduce a version of
Theorem~\ref{thm:CImaxstabledensity}, in which more than two
subvectors are considered.

\begin{corollary}\label{cor:CImaxstabledensity}
  Let $X=(X_i)_{i \in I}$ be a simple max-stable random vector with
  positive continuous density. Then, for any disjoint
  non-empty subsets $A_1,\dots,A_k$ of $I$, the conditional independence
  \mbox{$\ci_{i=1}^{k} X_{A_i} \mid X_{I \setminus \bigcup_{i = 1}^k A_i}$}
  implies the independence \mbox{$\ci_{i=1}^{k} X_{A_i}$}. 
\end{corollary}

(e) The non-degenerate univariate max-stable laws are classified up to
location and scale by the one parameter family of extreme value
distributions indexed by $\gamma \in \RR$
\begin{align*}
  F_\gamma(x)=\exp(-(1+\gamma x)^{-1/\gamma}), \qquad x \in \left\{
    \begin{array}{ll}
      (-1/\gamma, \infty) & \gamma > 0,\\
      \RR & \gamma = 0,\\
      (-\infty, -1/\gamma) & \gamma <0.
    \end{array} \right.
\end{align*}
Any other (not necessarily simple) max-stable random vector is
obtained through a transformation of the marginals that is
differentiable and strictly monotone on the respective sub-domain on
$\RR^d$ (cf.\ e.g. \cite{res08} Prop. 5.10).
Hence, the above results remain valid for the general class of
max-stable random vectors.

(f) \cite{domeyi14} show that, up to time reversal, only
max-auto\-regressive processes of order one can appear as discrete
time stationary max-stable processes that satisfy the first order Markov
property. This result indicates already that the conditional
independence assumption is to some extent unnatural in presence of the
max-stability property.\\

\begin{example}\label{ex:examples} 
Various classes of tractable max-stable distributions admit a positive continuous density, such that Theorem~\ref{thm:CImaxstabledensity} and Corollary~\ref{cor:CImaxstabledensity} apply.
Popular models that are commonly used for statistical inference include the asymmetric logistic model \citep{tawn90}, the asymmetric Dirichlet model \citep{ct91}, the pairwise beta model \citep{cooletal10} and its generalizations involving continuous spectral densities \citep{ballschl11} in the multivariate case. Moreover, most marginal distributions of spatial models such as the Gaussian max-stable model \citep{smit90b,gentetal11} or the Brown-Resnick model \citep{kabletal09,hr89} possess a positive continuous density if the parameters are non-degenerate.
Hence, if any of the components of the previously mentioned extreme value models exhibit conditional independence given any of the remaining components, they must be independent.
\end{example}

In the remaining article we subsume auxiliary arguments in
Section~\ref{sec:exponentmeasures} and give all proofs in
Section~\ref{sec:proofs}.

\section{Preparatory results on max-stable random
  vectors} \label{sec:exponentmeasures}

Throughout this section let $G$ be the distribution function of a
simple max-stable random vector $X=(X_i)_{i \in I}$ that has a
positive continuous density. We denote its
exponent function by
\begin{align*}
  V(x):= -\log G(x)= \int_{S_+} \max_{i \in I}
  \left(\frac{\omega_i}{x_i}\right)\, H(\D \omega), \qquad x \in
  (0,\infty)^I.
\end{align*}
Lower order marginals $G^A$ that refer to a subset $A$ of $I$ are
obtained as $\min_{i \in A^c}(x_i) \to \infty$, where $A^c=I\setminus A$.
We write $x_{A^c} \to \infty$ for $\min_{i \in A^c}(x_i) \to \infty$, 
and with this notation
\begin{align*}
  G^A(x_A):= \lim_{x_{A^c} \to \infty} G(x)
  \quad \text{and} \quad
  V^A(x_A):= -\log G^A(x_A).
\end{align*}
Since $G$ is absolutely continuous, the partial derivatives
\begin{align*}
  G^A_B(x_A):= \frac{\partial^{|B|}}{\partial x_{B}} G^A(x_A)
\quad \text{and} \quad
  V^A_B(x_A):= \frac{\partial^{|B|}}{\partial x_{B}} V^A(x_A)
\end{align*}
exist and are continuous for $B \subset A$, and the latter $V^A_B$ are homogeneous of order $-(|B|+1)$ \citep{ct91}. 
An elementary computation shows that
\begin{align*}
  G^{A}_B(x_{A}) &= {W^{A}_B(x_{A})} \exp\left(-V^{A}(x_{A})\right),  
\end{align*}
where
\begin{align*}
  W^N_M(x_M)= \sum_{\pi \in \Pi(M)} (-1)^{|\pi|} \prod_{J \in \pi}
  V^N_J(x_N), 
\end{align*}
and $\Pi(M)$ stands for the set of partitions of $M$
for $M \subset N \subset I$.

Let us further denote the set of non-empty subsets of $I$ by
$\subs(I)$. The collection of exponent functions $(V^A)_{A \in
  \subs(I)}$ is in a one-to-one correspondence with its
M\"obius inversion $(d_A)_{A \in \subs(I)}$, i.e., if we set
\begin{align*}
  d_A(x):= \sum_{B \in \subs(I) : A^c \subset B} (-1)^{|B \cap
    A|+1}V^B(x_B),
\end{align*}
it follows that $V^A$ can be recovered via
\begin{align}\label{eqn:moeb}
  V^A(x_A)=\sum_{B \in \subs(I) : B \cap A \neq \emptyset} d_B(x)
\end{align}
(cf.\ \cite{paptawn14}, Theorem~2 and \cite{schlathertawn_02}, Theorem~4 or, more generally, \cite{berge71} Chapter~3, Section~2 for the M\"obius inversion).
Finally, we define
\begin{align*}
  \chi_A(x_A)& := \lim_{x_{A^c} \to \infty}
  d_A(x)
  =\sum_{B \in \subs(I) : B \subset A} (-1)^{|B|+1}V^B(x_B) =\sum_{B
    \in \subs(I) : A \subset B} d_B(x).
\end{align*}
Then the collection of functions $(\chi_A)_{A \in \subs(I)}$ is also in a one-to-one
correspondence with $(V^A)_{A \in \subs(I)}$ as well as $(d_A)_{A \in
  \subs(I)}$ and the inversions are given by
\begin{align*}
  d_A(x)&= \sum_{B \in \subs(I): A \subset B} (-1)^{|B \setminus A|} \chi_{B}(x_B),\\
  V^A(x_A)&= \sum_{B \in \subs(I) : B \subset A} (-1)^{|B|+1}
  \chi_B(x_B).
\end{align*}
Further expressions for $V^A$, $d_A$ and $\chi_A$ are collected in Lemma~\ref{lemma:d_spectralmeasure}. Note that $\chi_A (x_A) \geq d_A(x)$ and thus,
\begin{align}\label{eqn:zeroequivalence}
  d_A = 0 \quad \Leftrightarrow \quad \chi_A = 0.
\end{align}

\begin{lemma}\label{lemma:d_spectralmeasure}
  The functions $V^A$ and $d_A$ and $\chi_A$ (with $A \in \subs(I)$)
  can be expressed in terms of the spectral measure $H$ as follows:
  \begin{align*}
    V^A(x_A)& =\int_{S_+} \max_{i \in A} \left(\frac{\omega_i}{x_i}\right)\, H(\D \omega),\\
    d_A(x) &= \int_{S^+} \left[ \min_{i \in A}
      \left(\frac{\omega_i}{x_i}\right) - \max_{j \in A^c}
      \left(\frac{\omega_j}{x_j}\right) \right]_+
    \, H(\D \omega),\\
    \chi_A(x_A) &= \int_{S^+} \min_{i \in A}
    \left(\frac{\omega_i}{x_i}\right) \, H(\D \omega).
  \end{align*}
  Here $z_+=\max(0,z)$ and $\max(\emptyset)=0$.
\end{lemma}

It turns out that the following two quantities are closely linked to
conditional independence and independence of subvectors of $X$,
respectively. For non-empty disjoint subsets $A,B$ of $I$ and
$C=I\setminus(A \cup B)$, we set for $x \in (0,\infty)^I$
\begin{align*}
  d_{A,B}(x) &:= V^{A \cup C}(x_{A \cup C}) + V^{B \cup C}(x_{B \cup C}) - V(x)-V^C(x_C)\\
  &\,= \sum_{L \in \subs(I): L \cap A \neq \emptyset, L \cap B \neq \emptyset, L \cap C = \emptyset} d_L(x),\\
  \chi_{A,B}(x_{A \cup B}) &:= \lim_{x_C \to \infty} d_{A,B}(x) = V^{A}(x_{A}) + V^{B}(x_{B}) - V^{A\cup B}(x_{A \cup B}) \\
  &\,= \sum_{L \in \subs(I): L \cap A \neq \emptyset, L \cap B \neq \emptyset} d_L(x),
\end{align*}
where the last equalities follow from \eqref{eqn:moeb}.
The following lemma is an analogue of Lemma~\ref{lemma:d_spectralmeasure}.

\begin{lemma}\label{lemma:d_extended_spectralmeasure}
  The functions $d_{A,B}$ and $\chi_{A,B}$ can be expressed in terms
  of the spectral measure $H$ as follows
  \begin{align*}
    d_{A,B}(x) &= \int_{S^+} \left[ \min \left(\max_{i \in A}
        \left(\frac{\omega_i}{x_i}\right), \max_{i \in B}
        \left(\frac{\omega_i}{x_i}\right) \right) - \max_{j \in C}
      \left(\frac{\omega_j}{x_j}\right) \right]_+
    \, H(\D \omega),\\
    \chi_{A,B}(x_{A \cup B}) &= \int_{S^+} \min \left( \max_{i \in A}
      \left(\frac{\omega_i}{x_i}\right), \max_{i \in B}
      \left(\frac{\omega_i}{x_i}\right) \right) \, H(\D \omega).
  \end{align*}
\end{lemma}

Note that $\chi_{A,B}(x_{A \cup B}) \geq d_{A,B}(x)$ implies similarly
to (\ref{eqn:zeroequivalence}) that
\begin{align}\label{eqn:extzeroequivalence}
  d_{A,B} = 0 \quad \Leftrightarrow \quad \chi_{A,B} = 0.
\end{align}

{General expressions for the regular conditional distributions for the
  distribution of a max-stable process conditioned on a finite number
  of sites that are based on hitting scenarios of Poisson point
  process representations have been computed in
  \cite{domeyi13,oestschl14,oest15} under mild regularity assumptions or in
  \cite{wansto11} for spectrally discrete max-stable random vectors.}

Let again $A$ and $B$ be non-empty disjoint subsets of $I$.
Since we assumed a positive continuous density for $G$ (and hence also
for its marginals), the numerators and denominators in
\begin{align*}
  G(x_A|x_B)
:= \frac{G^{A \cup B}_B(x_{A \cup B})}{G^B_B(x_B)} = \exp\left(-\left[V^{A \cup B}(x_{A \cup B})-V^B(x_B) \right]
  \right) \frac{W^{A \cup B}_B(x_{A \cup B})}{W^B_B(x_B)}
\end{align*}
are non-zero and continuous for $x \in (0,\infty)^I$ and the expression $G(x_A | x_B)$ constitutes a regular version of the conditional probability $\PP(X_A \leq x_A | X_B = x_B)$.

\begin{proposition}\label{prop:independence}
The functions $\chi_{A,B}$ and $d_{A,B}$ are connected with the independence and conditional independence of the respective subvectors of $X$ as follows.
\begin{enumerate}[a)]
\item  $X_A \ci X_B \mid  X_{I\setminus (A \cup B)}\quad \Rightarrow \quad d_{A,B}=0$.
\item  $X_A \ci X_B \quad \Leftrightarrow \quad \chi_{A,B}=0$.
\end{enumerate}
\end{proposition}

\begin{remark}
The assumption that $G$ admits a positive continuous density on $(0,\infty)^I$ is crucial for part a) to hold true. It fails in Example~\ref{ex:max}.
\end{remark}

Moreover, it is a simple consequence of \cite{ber61} and
\cite{dehaan78} that the pairwise independence of any disjoint
subvectors of the simple max-stable random vector $X$ implies already
their joint independence.

\begin{lemma}\label{lemma:pwindep}
  If $X_{A_1},\dots,X_{A_k}$ are pairwise independent subvectors of a
  simple max-stable random vector $X$ (for necessarily disjoint $A_i
  \subset I$), then they are jointly independent.
\end{lemma}

\section{Proofs} \label{sec:proofs}

\begin{proof}[Proof of Lemma~\ref{lemma:d_spectralmeasure}]
The first equation is clear from the definition of $V^A$.
The relation for $d_A$ can be obtained as follows.
\begin{align*}
d_A(x)
&= \sum_{B \in \subs(I) : A^c \subset B} (-1)^{|B \cap A|+1}V^B(x_B)\\
&= \int_{S_+} \sum_{B \in \subs(I) : A^c \subset B} (-1)^{\lvert B \cap A\rvert +1}  \max_{i \in B} \left(\frac{\omega_i}{x_i}\right)\, H(\D \omega)\\
&= \int_{S^+} 
\left[
\min_{i \in A} \left(\frac{\omega_i}{x_i}\right)
- \max_{j \in A^c} \left(\frac{\omega_j}{x_j}\right)
\right]_+
\, H(\D \omega).
\end{align*}
In order to obtain the last equality, we denote $a_i=\omega_i/x_i$ and
distinguish two cases:\\
1st case: $A=I$. Then
\begin{align*}
\sum_{B \in \subs(I) : A^c \subset B} (-1)^{\lvert B \cap A\rvert +1}  \max_{i \in B} \left(a_i\right)
= \sum_{B \in \subs(I)} (-1)^{\lvert B\rvert +1}  \max_{i \in B} \left(a_i\right) = \min_{i \in I} \left(a_i\right).
\end{align*}
2nd case: $A \neq I$. Then set $b:=\max_{i \in A^c} a_i$ and $c_i:=\max(a_i,b)$, such that 
\begin{align*}
&\sum_{B \in \subs(I) : A^c \subset B} (-1)^{\lvert B \cap A\rvert +1}  \max_{i \in B} \left(a_i\right)
= \sum_{B \in \subs(I) : A^c \subset B, B \neq A^c} (-1)^{\lvert B \cap A\rvert +1}  \max_{i \in B \cap A} \left(c_i\right) - b\\
&=  \sum_{U \subset A : U\neq \emptyset} (-1)^{\lvert U \rvert +1}  \max_{i \in U} \left(c_i\right) - b
= \min_{i \in A} \left(c_i\right) - b\\
&= \min_{i \in A} \left(\max(a_i,b)\right) - b
=  \max\left(\min_{i \in A}(a_i),b\right) - b
= \left(\min_{i \in A}(a_i) - b\right)_+ .
\end{align*}
The expression for $\chi_A$ follows immediately.
\end{proof}

\begin{proof}[Proof of Lemma~\ref{lemma:d_extended_spectralmeasure}]
Similar to the proof of Lemma~\ref{lemma:d_spectralmeasure}, 
the relation for $d_{A,B}$ follows from  
\begin{align*}
d_{A,B}(x) &= V^{A \cup C}(x_{A \cup C}) + V^{B \cup C}(x_{B \cup C}) - V(x)-V^C(x_C)\\
&= \int_{S_+} 
\max_{i \in A\cup C} \left(\frac{\omega_i}{x_i}\right)   
+ \max_{i \in B\cup C} \left(\frac{\omega_i}{x_i}\right)   
- \max_{i \in I} \left(\frac{\omega_i}{x_i}\right)   
- \max_{i \in C} \left(\frac{\omega_i}{x_i}\right)   
\, H(\D \omega)\\
&= \int_{S^+} \left[ \min \left(\max_{i \in A}
  \left(\frac{\omega_i}{x_i}\right), \max_{i \in B}
  \left(\frac{\omega_i}{x_i}\right) \right) - \max_{j \in C}
  \left(\frac{\omega_j}{x_j}\right) \right]_+
\, H(\D \omega),
\end{align*}
where the last equality is obtained from 
\begin{align*}
&\max_{i \in A\cup C} \left(a_i\right)   
+ \max_{i \in B\cup C} \left(a_i\right)   
- \max_{i \in I} \left(a_i\right)   
- \max_{i \in C} \left(a_i\right)   \\
&=
\min\left(\max_{i \in A\cup C} \left(a_i\right),
\max_{i \in B\cup C} \left(a_i\right)\right)
- \max_{i \in C} \left(a_i\right)   \\
&= \max\left [\min\left(\max_{i \in A} \left(a_i\right),
\max_{i \in B} \left(a_i\right)\right) , \max_{i \in C}\left(a_i\right)\right]
- \max_{i \in C} \left(a_i\right)   \\
&= \left [\min\left(\max_{i \in A} \left(a_i\right),
\max_{i \in B} \left(a_i\right)\right) - \max_{i \in C}\left(a_i\right)\right]_+
\end{align*}
if we denote $a_i=\omega_i/x_i$.
The expression for $\chi_{A,B}$ follows immediately.
\end{proof}

\begin{proof} [Proof of Proposition~\ref{prop:independence}]
\begin{enumerate}[a)]
\item  As before, let $C=I\setminus (A \cup B)$. Since $G(x)=\exp(-V(x))$ has a positive continuous density, we have that  the conditional independence 
\mbox{$X_A \ci X_B \mid X_C$} for $C=I \setminus (A \cup B)$ implies that for all $x \in (0,\infty)^I$
\begin{align*}
G(x_A | X_C) \, G(x_B | X_C) = G(x_{A \cup B} | X_C) \qquad \PP\text{-a.s.} \,.
\end{align*}
Since $X_C$ has a positive continuous density with respect to the Lebesgue-measure on $(0,\infty)^C$, it follows that
\begin{align*}
G(x_A | x_C) \, G(x_B | x_C) = G(x_{A \cup B} | x_C) \qquad  \text{for all } x \in Q,
\end{align*}
where $Q$ is a dense subset of $(0,\infty)^I$. By the continuity of these expressions in $x \in (0,\infty)^I$, the equality holds for all $x \in (0,\infty)^I$ and is equivalent to
\begin{align}\label{eqn:proofCImaxstabledensity}
\exp\left(d_{A,B}(x) \right)
= \frac{W^{A \cup C}_C(x_{A \cup C}) W^{B \cup C}_C(x_{B \cup C})}{W^{A \cup B \cup C}_C(x_{A \cup B \cup C}) W^C_C(x_C)}, \qquad x \in (0,\infty)^I.
\end{align}
Here, $d_{A,B} \geq 0$ and $d_{A,B}$ is homogeneous of order $-1$, while the components $V^N_J$ that build the terms $W^N_M$ are homogeneous of order $-(|J|+1)$.
Now, replacing $x$ by $t^{-1} x$ for $t > 0$ in (\ref{eqn:proofCImaxstabledensity}), we see that the left-hand side grows exponentially in the variable $t$ as $t$ tends to $\infty$ if $d_{A,B}(x)>0$, while the right-hand side exhibits at most polynomial growth. Therefore, $d_{A,B}(x) = 0$ for $x \in (0,\infty)^I$. 
\item Both sides are equivalent to $G^A(x_A) G^B(x_B) = G^{A \cup B}(x_{A \cup B})$ for all $x \in (0,\infty)^I$. \qedhere
\end{enumerate}
\end{proof}

\begin{proof}[Proof of Theorem \ref{thm:CImaxstabledensity}]
The hypothesis follows from Proposition~\ref{prop:independence} and  \eqref{eqn:extzeroequivalence}.
\end{proof}

\begin{proof}[Proof of Lemma~\ref{lemma:pwindep}]
It suffices to show that for $x_{A_i} \in (0,\infty)^{A_i}$, $i=1,\dots,k$ and $r \in (0,\infty)$
\begin{align*}
\PP\left(X_{A_1} \leq x_{A_1},\dots,X_{A_k}\leq x_{A_k}\right) 
= \prod_{i=1}^k\PP\left(X_{A_i} \leq x_{A_i}\right).
\end{align*}
Using the notation $r_i=\sum_{j_i \in A_i} x^{-1}_{j_i}$, $u_{j_i}=(r_ix_{j_i})^{-1}$ for $j_i \in A_i$ and $Y_{i}=\max_{j_i \in A_i} u_{j_i} X_{j_i}$, $i=1,\dots,k$, we can rewrite this equality in the form
\begin{align*}
\PP\left(Y_1 \leq r_1^{-1},\dots, Y_k \leq r_k^{-1} \right) 
= \prod_{i=1}^k\PP\left(Y_i \leq r_i^{-1}\right),
\end{align*}
where the random vector $(Y_1,\dots,Y_k)$ is simple max-stable \citep{dehaan78} and has pairwise independent components due to our assumptions. Hence, by \cite{ber61} Theorem~2, the $Y_i$ are jointly independent, which entails the relation above.
\end{proof}

\begin{proof}[Proof of Corollary~\ref{cor:CImaxstabledensity}]
\mbox{$\ci_{i=1}^{k} X_{A_i} \mid X_{I \setminus \bigcup_{i = 1}^k A_i}$} implies 
\mbox{$X_{A_{i_1}} \ci X_{A_{i_2}} \mid X_{I \setminus \bigcup_{i = 1}^k A_i}$} for $i_1 \neq i_2$ and hence 
\mbox{$X_{A_{i_1}} \ci X_{A_{i_2}}$} by Theorem~\ref{thm:CImaxstabledensity}.
The hypothesis follows if we apply Lemma~\ref{lemma:pwindep} to the $X_{A_i}$, $i=1,\dots,k$.
\end{proof}

{\footnotesize
 \paragraph{Acknowledgments}
The authors would like to thank an anonymous referee for helpful comments improving the clarity of the paper. IP acknowledges funding from the SuSTaIn program - Engineering and Physical Sciences Research Council grant EP/D063485/1 - at the School of Mathematics of the University of Bristol. Part of this paper was written during a visit of KS at the University of Bristol funded by the SuSTaIn program and during a visit of IP at the University of Mannheim funded by the German Research Foundation DFG through the RTG 1953. IP and KS thank them for their generous hospitality.

\bibliographystyle{agsm}

}

\end{document}